\documentclass[12pt,reqno]{amsart}
\usepackage{amsthm, amsfonts,  amsmath, amssymb}
\usepackage{url}
\usepackage[a4paper, top = 1.2in, bottom = 1.2in, left = 1.2in, right = 1.2in]{geometry}
\newtheorem{thm}{Theorem}[section]
\newtheorem{lem}[thm]{Lemma}
\newtheorem{prop}[thm]{Proposition}
\theoremstyle{definition}

\newtheorem{remark}[thm]{Remark}
\theoremstyle{plain}

\numberwithin{equation}{section}

\newcommand{\C}{\mathbb{C}}

\newcommand{\N}{\mathbb{N}}
\newcommand{\Q}{\mathbb{Q}}

\newcommand{\R}{\mathbb{R}}

\newcommand{\Z}{\mathbb{Z}}

\newcommand{\SL}{\mathrm{SL}}

\newcommand{\mrm}[1]{\mathrm{#1}}

\title[On the gaps between non-zero Fourier coefficients of CM eigenforms]{On the gaps between non-zero Fourier coefficients of eigenforms with CM}
\author[S. Kaushik]{Surjeet Kaushik}
\email{ma15resch01001@iith.ac.in}
\address{
Department of Mathematics \\
Indian Institute of Technology Hyderabad\\
Kandi, Sangareddy - 502285\\
INDIA. 
}
\author[N. Kumar]{Narasimha Kumar}
\email{narasimha.kumar@iith.ac.in}
\address{
Department of Mathematics \\
Indian Institute of Technology Hyderabad\\
Kandi, Sangareddy - 502285\\
INDIA. 
}

\date{}
\begin{document}
\begin{abstract}
Suppose $E$ is an elliptic curve over $\Q$ of conductor $N$ with complex multiplication (CM) by $\Q(i)$, 
and $f_E$ is the corresponding cuspidal Hecke eigenform in $S^{\mrm{new}}_2(\Gamma_0(N))$. 
Then $n$-th Fourier coefficient of $f_E$ is non-zero in the short interval $(X, X + cX^{\frac{1}{4}})$ 
for all $X \gg 0$ and for some $c > 0$. As a consequence, we produce infinitely many cuspidal CM eigenforms 
$f$ level $N>1$ and weight $k > 2$ for which $i_f(n) \ll n^{\frac{1}{4}}$  holds, for all $n \gg 0$.
% Also, we provide a sufficient condition under which every non-zero element of 
% $f \in S^{\mrm{new}}_{2k}(\Gamma_0(N))(k>1,N>1)$  to satisfy $i_f(n) \ll n^{\frac{1}{4}}$ for all $n \gg 0$.

\end{abstract}
  \maketitle

\maketitle{}
\section{introduction}
Estimating the size of possible gaps between the non-zero Fourier coefficients of modular cuspforms
is one of the fundamental and interesting object of study in number theory.  
In this short note, we are interested in a question of Serre in bounding the maximum length of consecutive zeros of 
elliptic curves with CM by $\Q(i)$ and cusp forms of higher weight and level.

A famous conjecture of Lehmer predicts that $\tau(n) \neq 0$ for any $n \geq 1$, where
$$ \Delta(z) =  \sum_{n=1}^{\infty} \tau(n) q^n$$ 
is the unique normalized cuspidal eigenform of weight $12$ and level $1$. 
In a relation to Lehmer's conjecture, in~\cite{Ser81}, Serre initiated the general study of 
estimating the size of possible gaps between the non-zero Fourier coefficients of modular cuspforms by the function
$$i_f(n) := \mrm{max}\ \{i: a_f(n+j)=0 \ \mrm{for\ all}\  0 \leq j \leq i\}.$$
In fact, he proved that if $f(z)$ is a cusp form of weight $k \geq 2$ which is not a linear combination of forms 
with CM, then 
\begin{equation*}
i_f(n) \ll n. 
\end{equation*}
In that article, he also poses a question, if the bound can be improved to $n^{\delta}$ with $0<\delta<1$. 
In~\cite{BO01}, Balog and Ono proved that, for a cuspidal eigenform on $\Gamma_0(N)$, not a linear combination of forms with CM, that for any $\epsilon > 0$,
\begin{equation*}
% \label{IneqBO}
i_f(n) \ll_{f,\epsilon} n^{17/41+\epsilon} \quad \forall\  n\in \N. 
\end{equation*}
Currently, the best bound for $i_f(n)$ is available  due to Kowalski, Robert, and Wu 
and they proved that for any holomorphic non-CM cuspidal eigenform $f$ on general congruence groups, 
\begin{equation*}
% \label{IneqKRW}
i_f(n) \ll_{f} n^{7/17+\epsilon},
\end{equation*}
holds for all $n \in \N$ (cf.~\cite{KRW07}). However, in the case of elliptic curves,
the best bound is available due to Alkan. In~\cite{Alk03}, he showed that, if $f$ is a weight $2$ cusp form corresponding to 
an elliptic curve without CM,
then
\begin{equation}
i_f(n) \ll_{f,\epsilon} n^{\frac{69}{169}+\epsilon}.
\end{equation}

Through various approaches, many mathematicians have contributed in answering this question. The approaches are either using Rankin-Selberg estimates,
or Chebotarev density theorem, or distribution of $B$-free numbers, etc., 
(cf.\cite{BO01},
\newline \cite{Alk03},\cite{Alk05},\cite{AZ05}, \cite{AZ05JNT}, \cite{Mat12}, \cite{KRW07}, \cite{AZ08}). 

% For example, the approach  to the non-vanishing problem through
% the distribution of $B$-free numbers has been considered by Alkan and Zaharescu (), Matom\"aki (\cite{Mat12}), and Kowalski, Robert, and Wu (\cite{KRW07}),
% and many more. 
%For more details on the other approaches and the relevant literature on this problem, we urge the reader to look 
%at the beautiful introductions of~\cite{BO01},~\cite{KRW07}.  
% It is interesting to note that there is another aspect of understanding $i_f(n)$, 
% that is, through the study of averages for $i_f(n)$ (cf.~\cite{AZ07}).

In~\cite{AZ05IJNT}, for the first time, the authors have exploited  the idea of using congruences to study $i_f(n)$.
There it was shown that 
$i_{\Delta}(n) \ll n^{\frac{1}{4}}$, where $\Delta$ is the unique normalized cuspidal eigenform of weight $12$. 
The proof of this theorem relies on the existance of sums of squares in short intervals of the form $(x, x+x^{\frac{1}{4}})$.
In~\cite{DG14}, the authors extended the above idea 
%on the existance of sums of squares in short intervals of the form $(x, x+cx^{\frac{1}{4}})$, along with the congruences for eigenvalues of level $1$ Hecke eigenforms, 
to show that for any non-zero modular form $f \in S_{k}(\Gamma_0(1))$ 
with $k \geq 12$, one has 
\begin{equation}
\label{basicidentity}
i_f(n) \ll n^{1/4} 
\end{equation}
$\forall \ n \gg 0,$ where the implied constant depends only on $k$. 
% By the Exponent pair hypothesis, the exponent $\frac{1}{4}$ is the best bound that one can expect in~\eqref{basicidentity}.  
%For more details, we refer the reader to the discussion after~\cite[Theorem 2]{DG14}. 

If the level $N>1$, then there are no similar general results are available with $i_f(n) \ll n^{\frac{1}{4}}$. 
However, in~\cite{DG15}, the authors were able to produce infinitely many non-isogenous elliptic curves for which~\eqref{basicidentity} holds.
In~\cite{Kum16}, the second author had constructed examples of non-CM (and CM) cuspidal eigenform of level $N>1$ and weight $k>2$ 
for which the~\eqref{basicidentity} holds.
% In this short note, we construct infinitely examples of CM cuspidal eigenform of level $N>1$ 
% and weight $k > 2$ for which the~\eqref{basicidentity} holds.

In the approaches through Chebotarev density theorem, or distribution of $B$-free numbers, etc..,
one need to assume that the cuspidal eigenforms are without CM.
This is because, all those approaches depends on a result of Serre, which is valid only for eigenforms without CM.  
However, the analogous question make sense for cuspidal eigenforms with CM as well, and this can be handled
through the approach of ~\cite{AZ05IJNT},~\cite{DG14},~\cite{DG15},~\cite{Kum16}.

% Still, we dont know how to construct infinitely many such non-CM eigenforms.

% In this note, we look at the elliptic curves over $\Q$ with complex multiplication. 
In this short note, we show that 
any elliptic curve over $\Q$ with CM by $\Q(i)$
satisfy
\begin{equation*}
i_f(n) \ll n^{\frac{1}{4}} 
\end{equation*}
for $n \gg 0$. Then, we shall produce a one-parameter family of elliptic curves 
over $\Q$ with CM by $\Q(i)$ and infinitely many cuspidal CM eigenforms level $N>1$ 
and weight $k > 2$ for which the~\eqref{basicidentity} holds.
% Later, we shall define the notion of highly congruent and provide a sufficient
% condition under which every non-zero element $f$ of $S^{\mrm{new}}_{2k}(\Gamma_0(N))(k>1, N>1)$ satisfy~\eqref{basicidentity}. 
% We provide an example supporting this theorem.

% \section{Acknowledgements}
% We would like to thank Dr. Sanoli Gun whose question on CM forms made us to think about this note.
% The basic idea of this note stems from the work of Das and Ganguly~\cite{DG14}. This work 
% was supported by IIT Hyderabad through Institute's startup research grant.

\section{Statements and proofs}
Suppose $E$ is an elliptic curve over $\Q$ of conductor $N$ with CM by $\Q(i)$.
Let $f_E \in S^{\mrm{new}}_2(\Gamma_0(N))$ be the corresponding cuspidal eigenform 
of level $N$, by the modularity theorem. In this section, we shall show that 
the elliptic curve $E$ satisfies
\begin{equation}
i_{f_E}(n) \ll n^{\frac{1}{4}} 
\end{equation}
for $n \gg 0$.

Let us briefly recall the notion of CM. Let $K$ be an imaginary quadratic field and $\mathcal{O}_K$ be the integral closure.
Let $\mathfrak{m}$ be an integral ideal of $K$ and let $I(\mathfrak{m})$ be the group of fractional ideals of $K$ co-prime to $\mathfrak{m}$. 
By definition, a Hecke character $\Psi$ of $K$ is a homomorphism $$\Psi: I(\mathfrak{m}) \to \C^*,$$
such that $\Psi(\alpha)=\alpha^r$ for all $\alpha\in K^*$ with $\alpha\equiv 1 \mod \ \mathfrak{m}$.
For such a Hecke character $\Psi$, one can associate the function $f=f_{\psi}$ defined by $$f_\psi(z)=\sum_{\mathfrak{a}\subseteq \mathcal{O}_K, (\mathfrak{a},\mathfrak{m})=1} 
\Psi(\mathfrak{a})e^{2\pi i(N\mathfrak{a})z},$$ where $N\mathfrak{a}$ denotes 
the norm of $\mathfrak{a}$. We can also write the function as $$f_\psi(z)=\sum_{n\geq1} a_{\psi}(n)e^{2\pi inz},$$ 
where \begin{equation}
\label{inertzero}
    a_{\psi}(n)=\sum_{(\mathfrak{a},\mathfrak{m})=1, N\mathfrak{a}=n} \Psi(\mathfrak{a}) 
\end{equation}
By a theorem of Hecke, $f_\psi \in S_{r+1}(\Gamma_0(|d_K|N\mathfrak{m}), \epsilon)$ is an eigenform,
where $d_K$ is discriminant of field $K$, $\epsilon$ is a character modulo $|d_K|N\mathfrak{m}$ (cf. refer to~\cite{Rib77} for more details).
Let us recall some useful results from ~\cite[Theorem 1]{DG14} and~\cite[Lemma 2.2]{KRW07}, 
which we shall use.
\begin{thm}[Das-Ganguly]
\label{sumsofsquares}
Given any integer $N \in \N$, there exists $X_0 \in \R^{+}$ and $c>0$ (depending only on $N$) such that 
there exists an integer $n$ which is a sum of two squares and co-prime to $N$ in intervals
of type $(X, X+cX^{\frac{1}{4}})$ for all $X \gg X_0$.
\end{thm}
\begin{lem}[Kowalski-Robert-Wu]
\label{KRWlemma}
If $f = \sum_{n=1}^{\infty} a_f(n) q^n$ is a  normalized cuspidal eigenform in $S_{2k}(\Gamma_0(N),\chi)$, then there exists a natural number $M_f \geq 1$ such
that for any prime $p \nmid M_f$, either $a_f(p) = 0$ or $a_f(p^r) \neq 0$ for all $r \geq 1$. If $\chi$ is trivial and
$f$ has integer coefficients, then one can take $M_f = N$. 
\end{lem}
Now, let us recall Deuring's theorem for elliptic curves over $\Q$  (cf. ~\cite[Chapter 13, Theorem 12]{Lan73}).
\begin{thm}[Deuring]
\label{DeuringTheorem}
Let $E$ be an elliptic curve over $\Q$ with CM by $K$.
Let $\Psi$ be the corresponding Hecke character. If $p \geq 5$, then
the number $a_{E}(p)$ is zero if and only if either $p$ is inert
or ramified in $K$ or divides the conductor of $\Psi$ (equivalently, 
the elliptic curve has bad reduction).  
\end{thm}

Now, we are ready to prove the theorem that we are alluded to in the beginning of this section.
\begin{thm}
\label{ellipticmaintheorem}
Let $E$ be an elliptic curve over $\Q$ of conductor $N$ with CM by $\Q(i)$
and $f \in S^{\mrm{new}}_2(\Gamma_0(N))$ be the corresponding cuspidal eigenform.
Then,
$$ i_{f}(n) \ll n^{\frac{1}{4}} $$
for $n \gg 0$, where the implied constant depends on $E$.
\end{thm}

\begin{proof}
Since $E$ is an elliptic curve over $\Q$ with CM, let $\Psi$ be the corresponding Hecke character.
Take $N = 6M_f\mrm{cond}(\Psi)$, where $M_f$ as in Lemma~\ref{KRWlemma}. By Theorem~\ref{sumsofsquares}, there exists $X_0 \in \R$ and
 $c>0$ (depending only on $N$) such that there exists an integer, say $m$, which is a sum of two squares 
 and co-prime to $N$ in intervals of type $(X,X+cX^{\frac{1}{4}} )$ for all $X \gg X_0$. So, to prove the theorem, it is sufficient
 to show that $a_f(m)$ is non-zero.

Since $m$ is a sum of squares and $(m,6)=1$, $m$ can be written as
$$m= \underset{p_i \equiv 1 \pmod 4}{\Pi} p_i^{r_i} \underset{q_i \equiv 3 \pmod 4}{\Pi} q_i^{2s_i}.$$

By~\eqref{inertzero}, we see that for inert primes $q$ of $\Q(i)$, the Fourier coefficients $a_f(q)$ are zero 
because there are no ideals of norm $q$. By quadratic reciprocity law, the odd primes $q$ which remain inert in $\Q(i)$ 
are exactly the primes $q \equiv 3 \pmod 4$. However, the power of $q_i$ in $m$ are even and 
we show that $a_f(q_i^{2s_i})$ is non-zero. This is because, the Hecke relations would imply that 
                         \begin{equation}
%                          \label{EC Mod4}
                         a_f(q_i^r) = a_f(q_i)a_f(q_i^{r-1}) - q_i a_f(q_i^{r-2}),  
                         \end{equation}
       would imply that $a_f(q_i^{2r}) = (-q_i)^r a_f(q_i^{2r-2})$. Since $a_f(q_i^2)$ is non-zero, 
       we see that $a_f(q_i^{2r})$'s are also non-zero,  for all $r \geq 1, i \geq 1$.
       
For any split prime $p$ of $\Q(i)$, the Fourier coefficient $a_f(p)$ is non-zero, 
by Theorem~\ref{DeuringTheorem} and $(m,6\mrm{cond}(\Psi))=1$.  
This implies that $a_f(p^r) \neq 0$ for all $r \geq 1$, since $(m,M_f)=1$ and by Lemma~\ref{KRWlemma}.
This shows that, for $p \equiv 1 \pmod 4$, we have $a_f(p^r) \neq 0$ for all $r \geq 1$, since
the odd primes $p \equiv 1 \pmod 4$ are exactly the split primes of $\Q(i)$.

Hence $$a_f(m) = \underset{p_i \equiv 1 \pmod 4}{\Pi} a_f(p_i^{r_i}) \underset{q_i \equiv 3 \pmod 4}{\Pi} a_f(q_i^{2s_i}) \neq 0,$$
hence we are done with the proof.
\end{proof}

\begin{remark}
The crux in the proof of Theorem~\ref{ellipticmaintheorem} is Deuring's Theorem.
We could have  a theorem similar to that of Theorem~\ref{ellipticmaintheorem} 
for eigenforms of weight $2k$ with $k>1$, if we had an analogous result of Deuring for higher weights. 
But, the authors are not aware of such results.
% there exists characterization of vanishing of Fourier coefficients in terms of the behavior of the primes,
% i.e., a similar characterization in the sense of Theorem~\ref{DeuringTheorem}.
\end{remark}
\begin{remark}
One might wonder the reason for working with $\Q(i)$, but not with any other imaginary quadratic fields.
The split (resp. inert) primes of $\Q(i)$ and the primes $p \equiv 1 \pmod 4$ (resp., $p \equiv 3 \pmod 4$)
occur in the decomposition of sums of squares have a relation. 
In fact, the primes $p \equiv 1 \pmod 4$ are exactly the split primes.
This fact, we have used it in the proof of the Theorem~\ref{ellipticmaintheorem}.
\end{remark}

There exists infinitely many elliptic curves $E/\Q$ with CM by $\Q(i)$. For example, 
one could take a one parameter family of elliptic curves $y^2=x^3+ax$, with $a$ varies over $\Q^*$. 
In this case, the endomorphism ring of $E$ is $\Z[i]$. However, these curves are isogenous,
since any two CM elliptic curves with the same endomorphism algebra are isogenous. 
% Peter L Clark Notes 

We shall finish this note with the following proposition.
\begin{prop}
There exists infinitely many eigenforms with CM of weight $k>2$ and level $N>1$ for which
~\eqref{basicidentity} holds.
\end{prop}
\begin{proof}
Let $E$ be an elliptic curve over $\Q$ with CM by $\Q(i)$. Let $\Psi$ be the corresponding Hecke character. 
Consider the Hecke character $\Psi^m$, for some odd $m$. By a result of Hecke, the corresponding $f_{\Psi^m}$ 
is an eigenform with CM of weight $m+1$ with trivial character. However, the eigenform $f_{\Psi^m}$ may not be a newform
unless $\Psi^m$ is primitive. 

By~\cite[Cor. 3.5]{Rib77}, corresponding to $f_{\Psi^m}$,
there exists a unique newform with CM, which we denote with $g_m= \sum_{n=1}^{\infty}{a_{g_m}(n)q^n}$. Then, the newform
$g_m$ of weight $m+1$ with trivial character and level dividing the level of $f_{\Psi^m}$. Moreover, $g_m(z)$ has the property
that $a_{g_m}(p) = a_{\Psi^m}(p)$ for primes $p$ away from the level of $f_{\Psi^m}$. Therefore, $a_{g_m}(p)=0$ if and only if $a_{\Psi^m}(p)=0$, for all but finitely many primes.
By~\cite[Prop. 5.1]{LK14}, if $p\geq 5$, we see that $a_{\Psi^m}(p)=0$ if and only if $a_{\Psi}(p)=0$.
By arguing as in the proof of Theorem~\ref{ellipticmaintheorem}, we see that 
$i_{g_m}(n) \ll  n^{\frac{1}{4}}$. Hence, there exists infinitely many eigenforms with CM of different
weights for which~\eqref{basicidentity} holds. The levels of $g_m$'s are $>1$ because eigenforms inside $S_k(\SL_2(\Z))$
are without CM, for any weight $k$.
\end{proof}

\bibliographystyle{plain, abbrv}

\end{document}